\documentclass{amsart}

\usepackage{url}

\usepackage{amscd}
\usepackage{amsfonts}
\usepackage{amssymb}
\usepackage{euscript}
\usepackage{amsmath}
\usepackage{amsthm}
\usepackage{MnSymbol}
\usepackage[matrix, arrow, curve]{xy}
\usepackage{mathrsfs}

\usepackage{tikz}
\usetikzlibrary{matrix}

\usepackage{blkarray}
\usepackage{tikz-cd}
\usepackage{dsfont}

\newcommand{\Ze}{\mathop{\mathrm{Z}}}
\newcommand{\ze}{\mathop{\zeta}}
\newcommand{\sq}[1]{\mathbf{#1}}

\newcommand{\M}[1]{\mathcal{M}_{0, #1}}
\newcommand{\Md}[1]{\mathcal{M}^\delta_{0, #1}}

\newcommand{\dlog}{\mathop{d\, log}}

\newcommand\overmat[2]{
  \makebox[0pt][l]{$\smash{\overbrace{\phantom{%
    \begin{matrix}#2\end{matrix}}}^{\text{$#1$}}}$}#2}

\newcommand\undermat[2]{
  \makebox[0pt][l]{$\smash{\underbrace{\phantom{\begin{matrix}#2\end{matrix}}}_{\text{$#1$}}}$}#2}

\theoremstyle{plain}
\newtheorem{prop}{Proposition}
\newtheorem{theorem}{Theorem}

\newtheorem*{cor}{Corollary}

\theoremstyle{definition}
\newtheorem{definition}{Definition}
\newtheorem{example}{Example}

\newenvironment{thmbis}[1]
  {\renewcommand{\thetheorem}{\ref{#1}$'$}%
   \addtocounter{theorem}{-1}%
   \begin{theorem}}
  {\end{theorem}}

\makeatletter
\g@addto@macro\th@remark{\thm@headpunct{}}
\makeatother

\theoremstyle{remark}
\newtheorem{proc}{Procedure}

\begin{document}

\title{On algebra of big  zeta values}

\begin{abstract}
The algebra of big zeta values  we introduce in this paper
is an intermediate object between multiple zeta values and  periods of the multiple zeta motive.
It consists of number series generalizing multiple zeta values, 
the simplest examples, which are not multiple zeta series, are Tornheim sums.
We show that  convergent big zeta values are periods of the moduli space of stable curves
of genus zero on one hand and multiple zeta values on the other hand.
It gives an alternative way to prove that
any such period may be expressed as a rational linear combination of multiple zeta values
and a simple algorithm for finding such an expression.
\end{abstract}

\author{Nikita Markarian}

\email{nikita.markarian@gmail.com}

\date{}

\thanks{The study has been funded within the framework of the HSE University Basic Research Program and the Russian Academic Excellence Project '5-100'.
}

\address{National Research University Higher School of Economics, Russian Federation,
Department of Mathematics, 20 Myasnitskaya str., 101000, Moscow,
Russia}

\maketitle

\section*{Introduction}
Multiple zeta values are number series
playing important role in a wide range of subjects,  see e.~g. \cite{B, GM, IKZ}.
Although the arithmetic nature of these numbers is still unknown,
one may try to consider all consequences of natural relations among
these series, that is the algebra of formal multiple zeta values. 
One could expect that there no other rational relations
among them, that is the formal algebra of multiple zeta values
is isomorphic to the algebra rationally generated by multiple zeta values.
The long-standing conjecture states that all relations
for the both of them are regularized double shuffle relations (\cite[Conjecture 1]{IKZ}).

In the first section, we introduce an algebra of number series called big zeta values, which generalizes
multiple zeta values. One may consider the corresponding formal algebra.
The main theorem of the paper states that relations among big zeta values imply that any big zeta value
equals a rational linear combination of multiple zeta values. Thus, the algebra of (formal)
big zeta values is another hopefully more convenient form
of the (formal) algebra of multiple zeta values.

In the second section, we describe in some detail relations among big zeta values.
An interesting purely algebraic problem is to find all
relations among formal multiple zeta values inside the formal algebra of big zeta values.
Stuffle relations obviously follow from these relations.
It is harder but possible to prove shuffle relations in terms of number series,
 see  \cite{KMT}.
One may expect that there are no other relations except regularized double shuffle
relations as it is suggested by the conjecture mentioned above.

The third section is devoted to the proof of the main theorem
of the paper, which states that every big zeta value
is a rational linear combination of multiple zeta values. Surprisingly enough,
the proof does not use all relations among big zeta values.
It does not use the invariance
of a big zeta value under a permutation of columns of the matrix defining the value.
Generally,  a basic matrix may become not basic after
such a permutation, except in the case of reflection. The invariance
of a big zeta value under reflection is analogous to the duality theorem (see e.~g. \cite{Dua})
for multiple zeta values. 

To prove the main theorem we build an algorithm
which allows us to turn any big zeta value 
into a rational linear combination of multiple zeta values. 
The algorithm is of independent interest. It may be considered
as a generalization of the stuffle product for the following reason.
The product of two multiple zeta  series is not a multiple zeta  series,
but a big zeta  series corresponding to the direct sum
of matrices of these multiple zeta values.
Applying the algorithm to it
we may expand it back in multiple zeta values. One may see that
the result is the stuffle product. 

In the last section,
we connect our construction with integrals  
of regular differential
forms with logarithmic singularities at
infinity on the moduli space of stable curves of genus 0 by the standard simplex
in cubical coordinates.
This construction motivated our definition of the algebra of big zeta values.
It turns out that term by term integration of the Taylor series
of such an integral is a rational linear combination of
big zeta series, and conversely, every big zeta series
is a rational linear combination of integrals of such series.

Different choices of coordinates on the integration region
result in different number series representing the same integral.
It gives a bunch of relations among the big zeta series generalizing the duality
relation mentioned above. It would be interesting to derive them
directly from basic relations.

The formal algebra of big zeta values  reminds the formal algebra of periods from  \cite{BCS, M} but is
more manageable. Relations in it mimic relations in the algebra of periods:
 Orlik--Solomon relations (\ref{osr}) correspond to the Stokes theorem 
and  harmonic product relations (\ref{hpr}) correspond to  Arnold relations. 
It would be interesting to find an analogous algorithm, which using integral relations expresses any integral as above as a rational combination of iterated integrals.
It would prove or disprove conjectures from \cite{BCS, M}.

Being proved that all relations among integrals as above
follows from basic relations among big zeta number series,
one may define purely algebraically the "integral" of regular differential
forms with logarithmic singularities at
infinity on the moduli space of stable curves of genus 0 by the standard simplex,
which takes values in the formal algebra of big zeta values.
This "integral" would have all properties of the usual integral
such as the Stokes theorem and the Fubini theorem.
It would follow that the  Drinfeld associator with coefficients in the formal algebra of  big zeta values defined by this "integral"
obey the standard associator relations.
Thus, the set of relations among formal multiple zeta values inside formal big zeta values algebra, which are conjecturally regularized double shuffle relations,
would imply the Drinfeld associator relations. This is the subject of future research.

{\bf Acknowledgments.} I am grateful to H.~Tsumura and W.~Zudilin for helpful discussions.

\section{Algebra of big  zeta values}

Denote by $e_{ab}$ a positive $A_w$-root,
that is an element $(r_i)\in\mathbb{Z}^w$
defined by
\begin{equation}
r_{i}=
\begin{cases} 1 & \mbox{if}\quad a\le i\le b \\
0 & \mbox{otherwise} \end{cases}
\label{root}
\end{equation}

\begin{definition}
A $d\times w$-matrix $A=(a_{ij})$ is called {\em basic} if it is of rank $d$, has no zero columns and
all its rows are positive $A_w$-roots  (\ref{root}). 
\label{def}
\end{definition}
In other words, rows of a basic matrix are coordinates  of positive $A_w$-roots,
which are linearly independent and are not contained in any coordinate hyperplane.

In the following definition the term "formal" means that we do not care
about convergence of the number series.

\begin{definition}
For a basic matrix $A$ the  {\em big zeta  series} is a formal series
\begin{equation}
\Ze(A)=\sum_{\substack{n_i\in \mathbb{N}, \\ 1\le i\le d}}\frac{1}{\prod_j(\sum_i a_{ij} n_i)}
\label{bzv}
\end{equation}
The width $w$ of $A$ is called the {\em weight}  and its height 
$d$ is called the {\em depth} of the big zeta series.
\end{definition}

Note that the big  zeta series does not depend on the order
of rows of the matrix.
\begin{example}
\label{e1}
$$
A=
  \begin{array}{@{} c @{}}
    \left (
      \begin{array}{ *{11}{c} }
        \overmat{k_d}{1 & \dots & 1} & \overmat{k_{d-1}}{1 & \dots & 1} & \dots & \dots&\overmat{k_1}{1 &\dots &1 } \\
	0 & \dots & 0 & 1 &\dots & 1 & \dots & \dots & \dots&  \dots & 1 \\
	0 & \dots & 0 & 0 &\dots & 0 & 1 & \dots & \dots&  \dots & 1 \\
        \vdots & \vdots & \vdots & \vdots &\vdots & \vdots& \vdots & \vdots & \vdots &  \vdots & \vdots \\
        0 & \dots & 0 & 0 &\dots & 0 & \dots & 0 & 1 &  \dots & 1 \\
      \end{array}
    \right )\\
    \mathstrut
  \end{array}
$$
\medskip

\begin{equation}
\Ze(A)=\sum_{\substack{n_i\in \mathbb{N}, \\ 1\le i\le d}}\frac{1}{ {n_1}^{k_d}(n_1+n_2)^{k_{d-1}}\cdots(n_1+\dots+n_d)^{k_1}}
=\ze(\sq{k_1, \dots, k_d})
\label{mzv}
\end{equation}

This is the multiple zeta series, see e.~g. \cite{IKZ}.
\end{example}

\begin{example}
\bigskip
\medskip
$$A=
  \begin{array}{@{} c @{}}
    \left (
      \begin{array}{ *{9}{c} }
		 \overmat{a+b}{1 & 1 & \dots & 1& \dots& 1} & 0 & \dots  & 0\\
         0& \dots & 0 & \undermat{b+c}{1 & \dots&1 & \dots &1&1}
      \end{array}
    \right )\\
    \mathstrut
  \end{array}
$$
\medskip
\begin{equation}
\Ze(A)=\sum_{n,m \in \mathbb{N}}\frac{1}{n^am^c(n+m)^b}
\end{equation}

This is the Tornheim sum, see e.~g. \cite{BZ}.
\end{example}

Space of formal big zeta series is equipped with a product
given by the direct sum of matrices.
Define the {\em formal big zeta values algebra}
as the space rationally generated 
by formal big zeta  series with this product modulo natural relations.
These relations are Orlik--Solomon relations (\ref{osr})
and harmonic product relations (\ref{hpr}) below
plus the invariance of the big zeta series under permutations
of rows and columns of the defining matrix.

Example \ref{e1} above shows that formal multiple zeta values lie
in the formal big zeta values algebra. The following theorem
states that they generate the whole algebra as a vector space.

\begin{theorem}
Any formal big zeta value  is a rational linear combination of formal multiple zeta values of the same weight.
\label{main}
\end{theorem}

The proof of Theorem \ref{main} occupies Section \ref{proof} below.

\begin{cor}
The sum of a convergent big zeta series is a linear rational combination
of multiple zeta values of the same weight.
\end{cor}

The following proposition gives a convergence criterion of
a big zeta series. It also follows from results of Section \ref{integral}
below.

\begin{prop}
For a basic matrix $A$  the big zeta series $\Ze(A)$
converges iff $A$ does not contain any rows with only one unit.
In other words, it converges iff the corresponding set of positive 
$A_w$-roots contains no simple roots.
\label{conv}
\end{prop}
\begin{proof}
Let $A$ contains a row with only one unit.
Consider the subseries, where  all summation parameters are fixed
except the one corresponding to this row.
This subseries is proportional to the harmonic series without some initial interval. 
Thus the series diverges.

The big zeta series (\ref{bzv}) is dominated by the series
$$
\sum_{\substack{n_i\in \mathbb{N}, \\ 1\le i\le d}}\frac{1}{\prod_j(\max_i a_{ij} n_i)}
$$
Divide the summation region in components corresponding to the diagonal
stratification of $\mathbb{N}^d$, that is to total orders on the set of rows.
 At each stratum this series
is equal to some multiple zeta series (\ref{mzv}), and $k_1$ of all these
series is more than the minimal number of units in rows of $A$.
If every row of $A$ contains at least two units, then $k_1$
of all these multiple zeta values are more than 1. All such multiple 
zeta series are known to be convergent, see e.~g. \cite{IKZ}.
It follows that the big zeta  series converges too.
\end{proof}

\section{Relations}

For  a basic matrix $A$  and a differential operator with constant coefficients
$D\in \mathbb{Q}[\partial/\partial z_1, \dots , \partial/\partial z_d]$ 
introduce notation:
\begin{equation}
\Ze(A, D)=\sum_{\substack{n_i\in \mathbb{N}, \\ 1\le i\le d}}\left(D\frac{1}{\prod_j(\sum_i a_{ij} z_i)}\right)(n_1, \dots, n_d)
\label{diff}
\end{equation}

\begin{prop}
For a basic $d\times w$-matrix $A$  and a homogeneous differential operator  with constant coefficients $D$, series $\Ze(A, D)$
is a rational linear combination of series $\Ze(A')$ for some $A'$s   of depth $d$ and of weight $w+\deg D$.
\end{prop}

\begin{proof}
If operator $D$ is of degree one, the right hand side of (\ref{diff})
is equal to a rational linear combination of series  $\Ze(A')$ ,
where $A'$ are matrices equal to $A$ with one  column doubled.
Then proof proceeds by induction on $\deg D$.
\end{proof}

Let $V$ be a vector space over $\mathbb{Q}$ and $\Delta\subset V^*$ be
a finite set of non-zero linear functions on $V$, which generates the dual space $V^*$. 
Denote by $G_\Delta$ the subspace of rational functions on $V$
generated by functions
$$
\frac{1}{\prod_{\alpha\in \kappa}\alpha^{n_\alpha}},
$$
where $\kappa\subset \Delta$ is a subset, which generates $V^*$.
The space $G_\Delta$ is a module over the algebra $\mathbb{Q}[V]$
 of differential operators with constant coefficients on $V$.
\begin{prop}
The $\mathbb{Q}[V]$-module $G_\Delta$ is generated by functions
$$
\frac{1}{\prod_{\alpha\in \kappa}\alpha},
$$
where $\kappa$ ranges over all subsets of $\Delta$, which are bases of $V^*$.
\label{vergne}
\end{prop}
\begin{proof}
The proof is straightforward, see \cite[Lemma 2]{BV}.
\end{proof}

\begin{prop}
For a basic $d\times w$-matrix $A$  and a  differential operator  with constant coefficients $D$,  series $\Ze(A, D)$ is a rational linear combination of series $\Ze(S, D\cdot D')$ for some $D'$s, where $S$ ranges over all square submatrices of $A$ of rank $d$.
\label{square}
\end{prop}

\begin{proof}
Let $\Delta$ be the set of linear functions given by columns of matrix $A$.
Restrictions imposed on this matrix by Definition \ref{def}
guarantee that conditions of Proposition \ref{vergne}
are satisfied. It follows that 
$$
\frac{1}{\prod_j(\sum_i a_{ij} z_i)} = \sum_{S}D'_{S} \frac{1}{\prod_j(\sum_i s_{ij} z_i)}
$$
for some operators  with constant coefficients $D'_S$,
where the summation is taken over all square submatrices $S=(s_{ij})$ of $A$
of rank $d$. Substituting this to  definition (\ref{diff})
we get the statement.
\end{proof}

Let $l_1, \dots,  l_k$ be non-zero linear functions on a vector space $V$
such that $\sum_i\alpha_i l_i=0$ for some numbers $\alpha_i$.
The  equality
\begin{equation}
\sum_{i=1}^k \,\frac{\alpha_i}{l_1\cdots \hat{l}_i\cdots l_k}=0
\label{OS}
\end{equation}
of rational functions on $V$ is referred to as the Orlik--Solomon relation
(\cite[3.5]{OT}).

\begin{prop}[Orlik--Solomon relations]
Let $A$  be a basic matrix  and $v_1, \dots, v_k$ be a set its distinct columns such that
$\sum_i \alpha_i v_i=0$ for some  $0\neq \alpha_i\in \mathbb{Q}$.
Then for a  differential operator  with constant coefficients $D$
\begin{equation}
\sum_{i=1}^k\alpha_i\Ze(\hat{A}_i, D)=0,
\label{osr}
\end{equation}
where $\hat{A}_i$ is  matrix $A$ with  column $v_i$ excluded.
\label{os}
\end{prop}

\begin{proof}
Substituting (\ref{OS}) in  definition (\ref{diff}) we get the statement.
\end{proof}

Another family of relations among big zeta series is 
given by subdivision of the set of summands of the series into  groups
with  different total orders of the summation arguments.

Introduce some notations. Consider maps of polynomial algebras
\begin{equation}
\begin{gathered}
m_1\colon \mathbb{Q}[z_1, z_2]\ni p(z_1,z_2)\mapsto p(z_1+z_2, z_1)\in \mathbb{Q}[z_1, z_2]\\
m_2 \colon \mathbb{Q}[z_1, z_2]\ni p(z_1,z_2)\mapsto p(z_1+z_2, z_2)\in \mathbb{Q}[z_1, z_2]\\
m_3 \colon  \mathbb{Q}[z]\ni p(z)\mapsto p(z_1+z_2)\in \mathbb{Q}[z_1, z_2]
\end{gathered}
\label{mmm}
\end{equation}
Denote by
\begin{equation}
\begin{gathered}
m_1^*\colon \mathbb{Q}[\partial/\partial z_1, \partial/\partial z_2]\to \mathbb{Q}[\partial/\partial z_1, \partial/\partial z_2]\\
m_2^*\colon \mathbb{Q}[\partial/\partial z_1, \partial/\partial z_2]\to \mathbb{Q}[\partial/\partial z_1, \partial/\partial z_2]\\
m_3^*\colon \mathbb{Q}[\partial/\partial z_1, \partial/\partial z_2]\to \mathbb{Q}[\partial/\partial z]
\end{gathered}
\label{m&ms}
\end{equation}
the maps linear dual to (\ref{mmm}).

For a set of $w-$vectors $\{v_i\}$ and a $d\times w$-matrix $M$ denote by $[v_1; \dots; v_n; M]$
the matrix $M$ with rows $v_i$ added somewhere.
Recall that a big zeta value does not depend on the order of rows.
Thus for a matrix and vectors as above
the value $\Ze([v_1; \dots; v_n; M], D)$ is well defined
for a basic matrix $[v_1; \dots; v_n; M]$,
the property of being basic also does not depend
on the order of rows.

\begin{prop}[Harmonic product relations]
Let $A$ be a basic matrix containing rows $e_{ij}$ and $e_{(j+1)k}$ (see (\ref{root})).
Denote by $A'$ the matrix $A$ with these rows excluded.
Then for a differential operator with constant coefficients $D$
acting on the space of rows of  $A$
\begin{equation}
\Ze(A, D)= \Ze([e_{ik};e_{ij};A'], m_1^*D)+ \Ze([e_{ik};e_{(j+1)k};A'], m_2^*D)+\Ze([e_{ik};A'], m_3^*D),
\label{hpr}
\end{equation}
where $m^*_i$ acts trivially on the subspace generated by rows of $A'$
and acts as in (\ref{m&ms}) on the subspace generated by the added  rows.
\label{hp}
\end{prop}

\begin{proof}
Denote by $n$ and $m$ the summation parameters corresponding
to rows $e_{ij}$ and $e_{(j+1)k}$
in $\Ze([e_{ij};e_{(j+1)k};A'], D)=\Ze(A,D)$.
Divide summands of this series in three groups:  ones with $n>m$,
ones with $n<m$ and ones with $n=m$.
One may see, that they correspond to three summands of the right hand side
of   (\ref{hpr}).
\end{proof}

\section{Proof of theorem \ref{main}}
\label{proof}
Denote by $T_d$ the upper triangular $d\times d$-matrix with units
on and above the diagonal.
One may see that for any homogeneous differential operators 
with constant coefficients $D_i$ of degree  $i$, the expression
\begin{equation}
\sum_{i=1}^w\Ze(T_i, D_{w-i})
\label{form}
\end{equation}
is a rational linear combination of multiple zeta  values (\ref{mzv}) of weight $w$.

On the other hand, by Proposition \ref{square}
any multiple zeta value of weight $w$ may be written in the form (\ref{form}).
It follows that Theorem \ref{main} may be reformulated as follows.

\begin{thmbis}{main}
Any formal big zeta value of weight $w$ is a rational linear combination of the ones of the form 
$\Ze(T_d, D_{w-d})$ for some  homogeneous differential operators 
with constant coefficients $D_i$ of degree  $i$.
\end{thmbis}

\begin{proof}
The proof proceeds by induction on the depth.
For depth $1$ the statement is clear. Suppose that it is proved for depths
less than $d$.

By Proposition \ref{square}, every big zeta value may be written
as a rational linear combination of series of the form $\Ze(S, D)$,
where $S$ is a basic $d\times d$-matrix. We need to show that all
these series are linear combinations of ones of the form $\Ze(T_d, D')$ modulo
series of a lower depth.  The strategy of the proof is to  replace iteratively
$\Ze(S, D)$ with linear combinations of
big zeta series whose matrices are closer and closer to $T_d$.

To define what means "closer", introduce an order on the set of indices of a $d\times d$-matrix as follows
$$
11\,<\,21\,<\,\dots\, <\,d1 \,<\,12\,<\,\dots\,<\, d2\,<\,\dots\, <\,(d-1)d\,< \,dd
$$
We say that two matrices are equal up to place $ij$ if 
they have equal entries at places not bigger than $ij$.
We say that two matrices are equal exactly up to the place $ij$
if they are equal up to this place and are not equal up to the next bigger place.
A matrix $A$ is closer to matrix $C$ than matrix $B$ if $A$ is equal to $C$
exactly up to place $ij$, $B$ is equal to $C$
exactly up to place $i'j'$ and $ij>i'j'$.

The proof consists in the iterative application of two procedures.

\begin{proc}
{\em turns a basic matrix into a linear combination of upper triangular
basic matrices.}

Given a basic matrix, 
rearrange its rows so that the number of the first nonzero entry in the row
decreases from the top to the bottom.
Consider rows with the first non-zero entries in the row. If there are more than one such rows, 
take any two  and apply Proposition \ref{hp} to them. We get three matrices.
The one of the lower rank may be thrown away by the induction assumption.
Take two other matrices.
Rearrange their rows. The number of rows with non-zero entry at the first column
of both matrices is one less than this number of the initial matrix. Repeat this operation until the number of such rows equals one.  
Then turn to  rows with the first non-zero element at the second place.
And so on.
\end{proc}

\begin{proc}
{\em turns an upper triangular basic matrix, which is equal to $T_d$
exactly up to place $(i-1)j$, into a linear combination of basic matrices, which are equal to $T_d$
up to   place $ij$.}

Given such a matrix,
let $k$ be the number of the last non-zero entry of its $j$-th row,
"$*$" means an unknown entry:
\renewcommand{\arraystretch}{1.8}
$$
\scalebox{0.75}{
\begin{blockarray}{cccccccccccccc}
 &  & & & &$i$ & & &$j$& &$k$ & & &\\
\begin{block}{c(ccccccccccccc)}
& 1 &1 & \dots& \dots&\dots &\dots&\dots &1 &* &\dots&\dots&\dots & *\\
& 0 &1 &\dots &\dots & \dots& \dots&\dots &1 &* &\dots&\dots &\dots& *\\
&\vdots &\vdots &\vdots &\vdots &\vdots &\vdots&\vdots &\vdots &\vdots &\vdots &\vdots&\vdots&\vdots \\
&0 &\dots &0 &1 &1 &\dots &\dots & 1& *&\dots &\dots&\dots& *\\
$i$&0 &\dots &0 &0 &1 &\dots &1& 0&\dots &\dots &\dots&\dots& 0\\
&0 &\dots &0 &0 &0 &1 &* & \dots&\dots &\dots &\dots&\dots& *\\
&\vdots &\vdots &\vdots &\vdots &\vdots &\vdots&\vdots &\vdots &\vdots&\vdots &\vdots & \vdots&\vdots\\
$j$&0 &\dots &\dots & \dots&\dots & \dots&0 &1 &\dots &1& 0&\dots&0\\
&\dots &\dots &\dots &\dots &\dots &\dots & \dots& \dots&\dots& \dots&\dots&\dots &\dots \\
\end{block}
\end{blockarray}
}
$$
Apply Proposition \ref{hp} to rows $i$ and $j$.
The matrix given by the last term in (\ref{hpr}) is of depth $d-1$ and may be thrown away by the induction
assumption. 
Rearrange the rows of two other matrices so that the number of the first nonzero entry 
in the row
decreases from the top to the bottom.

 The second term of the right hand side of (\ref{hpr}) 
gives a matrix, which
is equal to $T_d$ up to place $ij$.
Thus, consider the matrix given by the first  term.

Add to this matrix a column as follows, the added column is marked with an arrow:
\begin{equation}
\scalebox{0.75}{
\begin{blockarray}{cccccccccccccccc}
 &  & & & &$\mathbf\downarrow$&$i+1$ && & $j+1$&& &$k+1$ & & &\\
\begin{block}{c(ccccccccccccccc)}
& 1 &1 & \dots& \dots& 1 &\dots&\dots&\dots &1 &* &\dots&\dots&\dots&\dots & *\\
& 0 &1 &\dots & \dots& \dots&\dots& \dots&\dots &1 &* &\dots&\dots&\dots &\dots& *\\
&\vdots &\vdots &\vdots &\vdots&\vdots &\vdots &\vdots &\vdots &\vdots &\vdots &\vdots&\vdots&\vdots &\vdots&\vdots\\
&0 &\dots &0 &1&1 &1 &\dots&\dots & 1& *&\dots&\dots &\dots&\dots& *\\
$i$&0 &\dots &\dots &0 &1&1 &\dots&\dots & 1&\dots&\dots&1 &0&\dots& 0\\
$i+1$&0 &\dots &\dots &\dots&0 &1 &\dots&1 & 0&\dots&\dots &\dots &\dots&\dots& 0\\
&0 &\dots &\dots &\dots &\dots&0 &1 &\dots& *&\dots&\dots &\dots&\dots&\dots& *\\
&\vdots &\vdots &\vdots&\vdots &\vdots &\vdots&\vdots &\vdots &\vdots&\vdots &\vdots&\vdots &\vdots & \vdots&\vdots\\
$j-1$&0 &\dots &\dots &\dots &\dots& \dots&0&1 &* &\dots &\dots&\dots& \dots&\dots&*\\
$j$&0 &\dots &\dots &\dots &\dots &\dots&\dots & \dots& 0&1& *&\dots&\dots&\dots &* \\
&\dots &\dots &\dots &\dots &\dots &\dots&\dots & \dots&\dots& \dots&\dots& \dots&\dots&\dots &\dots \\
\end{block}
\end{blockarray}
}
\label{m}
\end{equation}
(it was an ambiguity in arranging $i$-th and $(i+1)$-th rows, we choose this variant).
One may see that columns with numbers from $i$ to $j+1$ in this matrix 
are linearly dependent, the first one is expressing through others.
Applying Proposition \ref{os} we get that the big zeta series with the matrix (\ref{m})
with the $i$-th column removed (that is the initial matrix)
is a linear combination of matrices, which are (\ref{m})
without some column numbered from $i+1$ to $j+1$.

Consider these matrices. Let the $i'$-th column was removed.
Then in the resulting matrix rows with numbers  $i'$ and $i'+1$
has $0$s before $i'$s place and $1$ at the $i'$s place.
As in Procedure 1, apply Proposition \ref{hp} to these rows and
rearrange the rows of the matrix 
so that the number of the first nonzero entry in the row decreases from top to bottom.
One may see that we get a matrix, which is equal to $T_d$ up to place $ij$.
\end{proc}

To prove the statement of the theorem, firstly apply Procedure 1 to matrix $S$. It gives an upper triangular matrix.
Find the first place, where it differs from $T_d$, denote it by $ij$. Apply Procedure 2.
The result is a matrix, which is closer to $T_d$. Apply to it Procedure
1. One may see that it 
does not change the first $i$ rows and the first $j$ columns of the matrix.
It follows that it is equal to $T_d$ at least up to place $ij$.
Find the next place where it differs from $T_d$. Again apply Procedure 2 and so on.
\end{proof}

\section{Integral representation}
\label{integral}

Denote by $\M{w+3}$ the moduli space of 
$w+3$ distinct points on the complex projective line 
considered up to the
action of the M\"obius group.
This group acts freely  on sets of three distinct points.
To introduce coordinates on $\M{w+3}$ 
fix three of $w+3$ points at $0$, $1$ and $\infty$.
Coordinates $(t_1, \dots, t_w)$ of other points of the configuration 
are called {\em simplicial coordinates} on $\M{w+3}$.

\begin{prop}
The algebra of regular differential
forms on $\M{w+3}$ with logarithmic singularities at
infinity
is generated by $1$-forms 
\begin{equation}
\omega_{ij}=\frac{dt_i-dt_j} {t_i-t_j} \quad 0\le i< j\le w+1 \quad ij\neq 0w+1,
\label{omegas}
\end{equation}
where $t_i$ are simplicial coordinates, $t_0=1$ and $t_{w+1}=0$.
The only relations among  them 
are   Arnold relations:
\begin{equation}
\omega_{ij}\wedge\omega_{jk}+\omega_{jk}\wedge\omega_{ik}+\omega_{ik}\wedge\omega_{ij}=0
\label{arr}
\end{equation}
\label{ar}
\end{prop}
\begin{proof}
See e.~g. \cite[6.1]{B}.
\end{proof}

Introduce cubical coordinates:
\begin{equation}
x_1=t_1 \quad x_2=t_2/t_1\quad \dots \quad x_w=t_w/t_{w-1}
\label{coord}
\end{equation}
The standard simplex 
\begin{equation}
\Delta_{w}=\{1>t_1>\dots>t_w>0,\quad t_i\in \mathbb{R} \}
\end{equation}
in cubical coordinates turns into the cube
\begin{equation}
\square=\{0<x_i<1,\quad x_i\in \mathbb{R}\}
\label{cube}
\end{equation}

Let $A=(a_{ij})$ be a basic $d\times w$-matrix such that the series $\Ze(A)$
converges, see Proposition \ref{conv}.
Integrating the Taylor series of the integrand term by term we get  a
relation 
\begin{equation}
\Ze(A)=\int_\square \left(\prod_{j=1}^d \frac{\prod_k x_k^{a_{jk}}}{1-\prod_k x_k^{a_{jk}}} \cdot \prod_{i=1}^{w} \frac{dx_i}{x_i}\right)
\label{int}
\end{equation}
where the integration region is given by (\ref{cube}).

\begin{prop}
The integrand of (\ref{int}) is a regular differential
form on $\M{w+3}$ with logarithmic singularities at
infinity.
\label{reg}
\end{prop}

\begin{proof}
Denote by $l(i)+1$ and $r(i)$ the numbers of the first and the
last unit in the $i$-th row of matrix $A$.
Substituting (\ref{coord}) into the integrand of (\ref{int})
we get the differential form
\begin{equation}
\prod_{j=1}^{d}\frac{t_{r(j)}}{t_{l(j)}-t_{r(j)}}\cdot\prod_{i=1}^{w}\frac{dt_i}{t_i}
\label{dform}
\end{equation}
where $t_0=1$.

Let us prove that this is a linear combination of products of differential forms (\ref{omegas}).

Denote by $G$ an unoriented graph with vertices $V(G)=\{0,1,\dots, w\}$,
with edges $E(G)$ corresponding to rows of $A$,
vertices $i-1$ and $j$ are connected by an edge iff $A$ contains  row $e_{ij}$
(see (\ref{root})). $G$ has no cycles, otherwise rows corresponding to edges
forming the cycle would be linear dependent. Thus, $G$ is a disjoint union of trees.
Make these trees rooted, choosing the vertex labeled by the minimal number
as a root.  Equip $G$ with an orientation by orienting any edge into direction "from the root"
of the tree to which this edge belongs. Call an edge "right" if its source's label 
is less than its target's one, otherwise call it  "wrong".

Rewrite (\ref{dform}) as
\begin{equation}
\prod_{\substack{(ij)\in E(G), \\(ij) \mbox{ \small is right}}} \frac{t_j}{t_i-t_j}
\cdot\prod_{\substack{(ij)\in E(G), \\(ij) \mbox{ \small is wrong}}} \left( \frac{t_i}{t_i-t_j}-1\right)\cdot\prod_{i=1}^{w}\frac{dt_i}{t_i}
\end{equation} 
where $(ij)$ connects vertices $i$ and $j$, $j>i$.

Expand brackets. Every vertex of $G$ has at most one incoming edge.
It follows that each of the summands has in the numerator
at most one $t_i$ for any $i$,
which cancels with the same term in the denominator.
Thus, each summand is 
proportional to a product of differential forms (\ref{omegas}).
\end{proof}

We left to the reader to check that for matrix satisfying conditions
of Proposition \ref{conv} the differential form given by this proposition is convergent on
the standard simplex,  convergence conditions are described in \cite{B, BCS}.

\begin{prop}
Differential forms appearing as integrands in (\ref{int}) 
for any, not only convergent basic matrix $A$, generate the space of regular differential
forms of top degree on  $\M{w+3}$ with logarithmic singularities at
infinity.
\end{prop}
\begin{proof}
Let
\begin{equation}
\bigwedge_{k=1}^{w} \omega_{i(k)j(k)}=\bigwedge_{k=1}^{w} \dlog(t_{i(k)}-t_{j(k)})
\label{factor}
\end{equation} 
 be a top degree form on  $\M{w+3}$ with logarithmic singularities at
infinity. By Proposition \ref{ar}, these forms generate
the space of top degree forms  with logarithmic singularities at
infinity. Consider the subset of indexes $k$ for which $j(k)\neq w+1$. Build a matrix with rows equal to roots
$e_{(i(k)+1)j(k)}$ for this set of indexes. These rows are linearly
independent, otherwise functions $(t_{i(k)}-t_{j(k)})$ would be 
linearly dependent, what would follow vanishing of the product of their
$\dlog$'s, this is  a consequence of the Arnlod relations (\ref{arr}).
Thus, the matrix is basic.

Consider the differential form, which is the integrand of (\ref{int})
associated with this matrix. By the proof of Proposition \ref{reg},
these form is equal to the initial one up to some 
differential forms of type (\ref{factor})  with a bigger number
factors $\omega_{i(w+1)}$. Applying  decreasing induction
on the number of such factors we get the statement.
\end{proof}

It follows that the space of convergent big zeta value of weight $w$ is the space of periods
of the pair $(\Md{w+3}, \Md{w+3}\setminus\M{w+3})$, notation $\Md{w+3}$
is introduced in \cite{B}.
Theorem \ref{main} states that this space is generated
by multiple zeta values of weight $w$.
It is in good agreement with the fact that all periods of these pair are given by periods of the fundamental group
of a projective line without three points (\cite{Del}),
that is multiple zeta values, what is proved in \cite{B}.
In contrast with this proof, which states that the weights of multiple
zeta values are not bigger than $w$, we get multiple zeta values exactly of weight  $w$.
It is natural because in (\ref{int}) we integrate a differential form 
with logarithmic poles.

\bibliographystyle{alpha}
\bibliography{bigzeta}

\begin{thebibliography}{KMT11}

\bibitem[BCS10]{BCS}
Francis Brown, Sarah Carr, and Leila Schneps.
\newblock The algebra of cell-zeta values.
\newblock {\em Compositio Mathematica}, 146(3):731–771, 2010.

\bibitem[Bro09]{B}
Francis C.~S. Brown.
\newblock Multiple zeta values and periods of moduli spaces
  {$\overline{\mathfrak {M}}_{0,n}(\mathbb{R})$}.
\newblock {\em Annales scientifiques de l'\'Ecole Normale Sup\'erieure}, Ser.
  4, 42(3):371--489, 2009.

\bibitem[BV99]{BV}
Michel Brion and Mich\`ele Vergne.
\newblock Arrangement of hyperplanes. {I}. {R}ational functions and
  {J}effrey-{K}irwan residue.
\newblock {\em Annales scientifiques de l'\'Ecole Normale Sup\'erieure}, Ser.
  4, 32(5):715--741, 1999.

\bibitem[BZ10]{BZ}
David~M. Bradley and Xia Zhou.
\newblock On {M}ordell-{T}ornheim sums and multiple zeta values.
\newblock {\em Ann. Sci. Math. Québec}, 34(1):15--23, 2010.

\bibitem[Del89]{Del}
Pierre Deligne.
\newblock Le groupe fondamental de la droite projective moins trois points.
\newblock In Y.~Ihara, K.~Ribet, and J.-P. Serre, editors, {\em Galois Groups
  over {$\mathbb{Q}$}}, volume~16, pages 79--297, New York, NY, 1989. Springer
  US.

\bibitem[GM04]{GM}
A.~B. Goncharov and Yu.~I. Manin.
\newblock Multiple {$\zeta$}-motives and moduli spaces
  {$\overline{\mathcal{M}}_{0,n}$}.
\newblock {\em Compositio Mathematica}, 140(1):1–14, 2004.

\bibitem[IKZ06]{IKZ}
Kentaro Ihara, Masanobu Kaneko, and Don Zagier.
\newblock Derivation and double shuffle relations for multiple zeta values.
\newblock {\em Compositio Mathematica}, 142(2):307–338, 2006.

\bibitem[KMT11]{KMT}
Yasushi Komori, Kohji Matsumoto, and Hirofumi Tsumura.
\newblock Shuffle products for multiple zeta values and partial fraction
  decompositions of zeta-functions of root systems.
\newblock {\em Mathematische Zeitschrift}, 268:993--1011, 2011.

\bibitem[Mar20]{M}
Nikita Markarian.
\newblock On generalized stuffle relations between cell-zeta values.
\newblock arXiv:2008.00855, 2020.

\bibitem[OT92]{OT}
Peter Orlik and Hiroaki Terao.
\newblock {\em Arrangements of Hyperplanes}, volume 300 of {\em Grundlehren der
  mathematischen Wissenschaften}.
\newblock Springer-Verlag Berlin Heidelberg, 1992.

\bibitem[SY19]{Dua}
Shin-ichiro Seki and Shuji Yamamoto.
\newblock A new proof of the duality of multiple zeta values and its
  generalizations.
\newblock {\em International Journal of Number Theory}, 15(06):1261--1265,
  2019.

\end{thebibliography}

\end{document}